\RequirePackage{fix-cm}
\documentclass[smallextended]{svjour3}       
\smartqed  
\usepackage{graphicx}
\usepackage{amsmath, amscd, amsfonts, amssymb, graphicx, color, enumerate}
\usepackage[bookmarksnumbered, colorlinks, plainpages]{hyperref}
\hypersetup{colorlinks=true,linkcolor=red, anchorcolor=green, citecolor=cyan, urlcolor=red, filecolor=magenta, pdftoolbar=true}
\newtheorem{prop}[theorem]{Proposition}

 \journalname{Bulletin of the Malaysian Mathematical Sciences Society}
\begin{document}

\title{Generalized Reverse Young and Heinz Inequalities \thanks{The author Shigeru Furuichi was partially supported by JSPS KAKENHI Grant Number 16K05257.}}


\author{Shigeru Furuichi \and Mohammad Bagher Ghaemi \and  Nahid Gharakhanlu.}

\authorrunning{ S. Furuichi \and M. B. Ghaemi \and N. Gharakhanlu} 

\institute{Shigeru Furuichi \at
              Department of Information Science, College of Humanities and Sciences, Nihon University, 3-25-40, Sakurajyousui, Setagaya-ku, Tokyo, 156-8550, Japan \\
              \email{furuichi@chs.nihon-u.ac.jp}           
           \and
            Mohammad Bagher Ghaemi \at
              School of Mathematics, Iran University of Science and Technology, 16846-13114, Narmak, Tehran, Iran\\
              \email{mghaemi@iust.ac.ir}
              \and
            Nahid Gharakhanlu \at
              School of Mathematics, Iran University of Science and Technology, 16846-13114, Narmak, Tehran, Iran\\
              \email{gharakhanlu\underline{
}nahid@mathdep.iust.ac.ir}
}

\date{Received: date / Accepted: date}

\maketitle

\begin{abstract}
In this paper, we study the further improvements of the reverse Young and Heinz inequalities for the wider range of $v$, namely  $v\in \mathbb{R}$. These modified inequalities are used to establish corresponding operator inequalities on a Hilbert space.
\end{abstract}

\section{Introduction}
The weighted arithmetic-geometric mean inequality, which is also called Young's inequality, states
\begin{equation*}
(1-v) a +v b \geq a^{1-v} b^v
\end{equation*}
for $a, b \geq 0$ and $v \in [0,1]$. If $v=\frac{1}{2}$, we obtain the arithmetic-geometric mean inequality
\begin{equation*}
\frac{a+b}{2}\geq \sqrt{ab}.
\end{equation*}
The Heinz means, introduced in \cite{bb}, are defined by
\begin{equation*}
H_{v}(a,b)=\frac{a^{1-v}b^{v}+a^{v}b^{1-v}}{2}
\end{equation*}
for $a,b\geq 0$ and $v\in [0,1]$. It is easy to see that
\begin{equation*}
\sqrt{ab}\leq H_{v}(a,b)\leq \frac{a+b}{2}\qquad v\in [0,1],
\end{equation*}
which is called Heinz inequality.

Let $B(H)$ denote the $C^{*}$-algebra of all bounded linear operators on a complex Hilbert space $H$. A self-adjoint operator $A\in B(H)$ is called positive, and we write $A\geq 0$ if $\langle Ax,x\rangle \geq 0$ for all $x\in H$.
The set of all positive operators is denoted by $B^{+}(H)$.
The set of all invertible operators in $B^{+}(H)$ is denoted by $B^{++}(H)$. We say $A\geq B$ if $A-B\geq 0$.

Let $A, B\in B^{++}(H)$ and $v\in [0,1]$. The $v$-weighted operator geometric mean of $A$ and $B$, denoted by $A\sharp_{v} B$, is defined as
\begin{equation*}
A\sharp_{v} B=A^{\frac{1}{2}}\left(A^{-\frac{1}{2}}B A^{-\frac{1}{2}}\right)^{v} A^{\frac{1}{2}}
\end{equation*}
and the $v$-weighted operator arithmetic mean of $A$ and $B$, denoted by $A\nabla_{v} B$, is
\begin{equation*}
A\nabla_{v} B= (1-v)A+vB.
\end{equation*}
When $v=\frac{1}{2}$, $A\sharp_{\frac{1}{2}}B$ and $A\nabla_{\frac{1}{2}} B$ are called operator geometric mean and operator arithmetic mean, and denoted by $A\sharp B$ and $A\nabla B$, respectively \cite{nn}. For positive operators $A, B$ and $v\in[0,1]$, we have \cite[Definition 5.2]{Mond}
\begin{equation*}
A\sharp_{v} B=B\sharp_{1-v} A.
\end{equation*}
It is well known that if $A, B\in B^{++}(H)$ and $v\in [0,1]$, then \cite{ee,ff}
\begin{align*}
A\nabla_{v} B \geq A\sharp_{v} B,
\end{align*}
which is the operator version of the scalar Young's inequality. An operator version of Heinz means was introduced in \cite{bb} by
\begin{equation*}
H_{v}(A,B)=\frac{A\sharp_{v} B+A\sharp_{1-v}B}{2},
\end{equation*}
where $v\in[0,1]$. In particular $H_{1}(A,B)=H_{0}(A,B)=A\nabla B$. It is easy to see that the Heinz operator means interpolate between the arithmetic and geometric operator means
\begin{equation*}
A\sharp B\leq H_{v}(A,B)\leq A\nabla B,
\end{equation*}
which is called the Heinz operator inequality \cite{ii,jj}.
We note that we use in Sect. 3 the following notations
$$A\natural_v B \equiv A^{\frac{1}{2}}\left(A^{-\frac{1}{2}}B A^{-\frac{1}{2}}\right)^{v} A^{\frac{1}{2}},\quad \hat{H}_v(A,B)\equiv \frac{A\natural_v B+A\natural_{1-v}B}{2},$$
for all $v \in \mathbb{R}$ including the range $v \notin [0,1]$.

Improvements of Young and Heinz inequalities and their reverses have been done for the weight $v\in [0,1]$ by many researchers. We refer the reader to \cite{AFK2015,aa,D2015_03,cc,FM2011,dd,xx,gg,hh,ii,jj,kk,ll,mm,qq,SC,SM,SS,oo} as a sample of the extensive use of Young and Heinz inequalities. One of the first refinements is as follows in Lemma \ref{l01} which one positive term was added to the right-hand side of the Young's inequality.

\begin{lemma}\textsc{(\cite{kk})}\label{l01}
Let $a,b\geq 0$ and $v\in[0,1]$. Then
\begin{align*}
(1-v)a+vb &\geq a^{1-v}b^{v}+r_{0}\left(\sqrt{a}-\sqrt{b}\right)^{2}
\end{align*}
where $r_{0}=\min\lbrace v,1-v \rbrace$.
\end{lemma}

However, in the recent paper \cite{oo}, Zhao and Wu provided two refining terms of Young's inequality in the following way.

\begin{lemma}\textsc{(\cite{oo})}\label{p1}
Let $a,b\geq 0$ and  $v\in[0,1]$. 
\begin{enumerate}[(i)]
\item  If $v \in [0,\frac{1}{2}]$, then
\begin{align*}
(1-v) a+v b \geq a^{1-v} b^v +v (\sqrt{a}-\sqrt{b})^2 +r_{0}(\sqrt{a}-\sqrt[4]{ab})^2,
\end{align*}

\item If $v \in [\frac{1}{2},1]$, then
\begin{align*}
(1-v) a+v b \geq a^{1-v} b^v + (1-v) (\sqrt{a}-\sqrt{b})^2 +r_{0}(\sqrt{b}-\sqrt[4]{ab})^2,
\end{align*}
\end{enumerate}
where $r=\min\lbrace v,1-v \rbrace$ and $r_{0}=\min\lbrace 2r,1-2r \rbrace$.
\end{lemma}

In the same paper, the following refined reverse versions have been proved too.
\begin{lemma}\textsc{(\cite[Lemma 2]{oo})}\label{l1}
Let $a,b\geq 0$ and  $v\in[0,1]$.
\begin{enumerate}[(i)]
\item If $v \in [0,\frac{1}{2}]$, then
\begin{align*}
(1-v)a+vb &\leq a^{1-v}b^{v}+(1-v)(\sqrt{a}-\sqrt{b})^{2}-r_{0}(\sqrt{b}-\sqrt[4]{ab})^{2},
\end{align*}
\item If $v \in [\frac{1}{2},1]$, then
\begin{align*}
(1-v)a+vb &\leq a^{1-v}b^{v}+v(\sqrt{a}-\sqrt{b})^{2}-r_{0}(\sqrt{a}-\sqrt[4]{ab})^{2},
\end{align*}
\end{enumerate}
where $r=\min\lbrace v,1-v \rbrace$ and $r_{0}=\min\lbrace 2r,1-2r \rbrace$.
\end{lemma}

Sababheh and Choi \cite{SC} obtained a complete refinement of the Young's inequality by adding as many refining terms as we like. For $a,b> 0$, $n \in \mathbb{N}$ and $v\in[0,1]$
\begin{align*}
(1-v)a+vb &\geq a^{1-v}b^{v}\\
&\quad + \sum_{k=1}^n s_{k}(v)
\left( \sqrt[2^k]{a^{2^{k-1}-j_{k}(v)} b^{j_{k}(v)}} -\sqrt[2^k]{a^{2^{k-1}-j_{k}(v)-1} b^{j_{k}(v)+1} }\right)^2,
\end{align*}
where $[x]$ is the greatest integer less than or equal to $x$ and
\begin{align*}
j_{k}(v)&=[2^{k-1}v],\\
r_{k}(v)&=[2^{k} v],\\
s_{k}(v)&=(-1)^{r_{k}(v)}2^{k-1}v+(-1)^{r_{k}(v)+1}\left[\frac{r_{k}(v)+1}{2}\right].
\end{align*}

Quite recently, Sababheh and Moslehian \cite{SM} gave a full description of all other refinements of the reverse Young's inequality in the literature as follows.

\begin{lemma}\textsc{( \cite[Theorem 2.1]{SM})}\label{SM}
Let $a,b> 0$ and $v\in[0,1]$.
\begin{enumerate}[(i)]
\item If $v \in [0,\frac{1}{2}]$, then
\begin{align*}
(1-v)a+vb &\leq a^{1-v}b^{v}+(1-v)(\sqrt{a}-\sqrt{b})^{2}-S_{n}(2v, \sqrt{ab}, b).
\end{align*}
\item If $v \in [\frac{1}{2},1]$, then
\begin{align*}
(1-v)a+vb &\leq a^{1-v}b^{v}+v(\sqrt{a}-\sqrt{b})^{2}-S_{n}(2(1-v), \sqrt{ab}, a).
\end{align*}
\end{enumerate}
Where 
\begin{align*}
S_{n}(v, a, b)&=\sum_{k=1}^n s_{k}(v) \left( \sqrt[2^k]{b^{2^{k-1}-j_{k}(v)} a^{j_{k}(v)}} -\sqrt[2^k]{a^{j_{k}(v)+1} b^{2^{k-1}-j_{k}(v)-1}}\right)^2,\\
j_{k}(v)&=[2^{k-1}v],\\
r_{k}(v)&=[2^{k} v],\\
s_{k}(v)&=(-1)^{r_{k}(v)}2^{k-1}v+(-1)^{r_{k}(v)+1}\left[\frac{r_{k}(v)+1}{2}\right].
\end{align*}
\end{lemma}

In the study of Young's inequalities, supplemental Young's inequality   
$$
a^v v^{1-v} \geq v a + (1-v) b
$$
for $a, b>0$ and $v \notin[0,1]$
is often discussed. 
Our main idea in this paper is to extend the range of $v$ and to give the tighter bounds of the reverse Young's inequalities proved in \cite{SM} and  \cite{oo}. In Theorem \ref{t2}, we will obtain a new generalization of the reverse Young's inequality which is stronger than the reverse Young's inequalities shown in \cite[Theorem 2.1]{SM} and \cite[Lemma 2]{oo}. Theorem \ref{t11} is another refinement of  \cite[Theorem 2.9]{SC} which extend the range of $v$. In Sect. 3, these modified inequalities are used to establish corresponding operator inequalities on a Hilbert space. We   emphasize that the significance of the inequalities in this paper is to have the wider range, namely $v\in \mathbb{R}$, and tighter bounds.

\section{Generalizations of the Reverse  Scalar Young and Heinz Inequalities}
In this section, we present the numerical inequalities needed to prove the operator versions. We start from the following  lemma to prove our main result.

\begin{lemma}\textsc{(\cite[Lemma 2.1]{aa})}\label{l0}
Let $a,b> 0$ and $v\notin [0,1]$. Then
\begin{align*}
(1-v)a+vb \leq a^{1-v}b^{v}.
\end{align*}
\end{lemma}

\begin{corollary}\label{c2}
Let $a,b> 0$ and $\frac{1}{2}\neq v\in \mathbb{R}$.
\begin{enumerate}[(i)]
\item If $v \notin [0,\frac{1}{2}]$, then
\begin{align*}
(1-v)a+vb &\leq a^{1-v}b^{v}+v(\sqrt{a}-\sqrt{b})^{2}.
\end{align*}

\item If $v \notin [\frac{1}{2},1]$, then
\begin{align*}
(1-v)a+vb &\leq a^{1-v}b^{v}+(1-v)(\sqrt{a}-\sqrt{b})^{2}.
\end{align*}
\end{enumerate}
\end{corollary}

\begin{proof}
\textit {(i)}  If $v \notin [0,\frac{1}{2}]$, then
\begin{align*}
&(1-v)a+vb-v(\sqrt{a}-\sqrt{b})^{2}\\
&=(1-2v)a+(2v)\sqrt{ab}\\
&\leq a^{(1-2v)} (\sqrt{ab})^{2v}\qquad \text{( by Lemma \ref{l0})}\\
&=a^{1-v}b^{v}. 
\end{align*}

\textit{(ii)}  If $v \notin [\frac{1}{2},1]$, then $(1-v) \notin [0,\frac{1}{2}]$. So by changing two elements $a,b$ and two weights $v,1-v$ in \textit {(i)}, the desired inequality is obtained.\qed
\end{proof} 

Next, we represent our main result which is the reverse Young's inequality for $v\in \mathbb{R}$.
\begin{theorem}\label{t2}
Let $a,b > 0$, $n \in \mathbb{N}$ such that $n\geq 2$ and $\frac{1}{2}\neq v\in \mathbb{R}$. Then,
\begin{enumerate}[(i)]
\item If $v \notin [\frac{1}{2},\frac{2^{n-1}+1}{2^n}]$, then
\begin{align}\label{eq5}
(1-v) a+v b &\leq a^{1-v} b^v+(1-v)(\sqrt{a}-\sqrt{b})^{2}\notag\\
&\qquad +(2v-1)\sqrt{ab}\sum_{k=2}^n2^{k-2}\left(\sqrt[2^{k}]{\frac{b}{a}}-1\right)^{2}.
\end{align}
\item If $v \notin [\frac{2^{n-1}-1}{2^n},\frac{1}{2}]$, then
\begin{align}\label{eq6}
(1-v) a+v b &\leq a^{1-v} b^v +v(\sqrt{a}-\sqrt{b})^{2}\notag\\
&\qquad +(1-2v)\sqrt{ab}\sum_{k=2}^n2^{k-2}\left(\sqrt[2^{k}]{\frac{a}{b}}-1\right)^{2}.
\end{align}
\end{enumerate}
\end{theorem}

\begin{proof} 
\textit{(i)} If $v \notin [\frac{1}{2},\frac{2^{n-1}+1}{2^n}]$, we have $(2^{n}v-2^{n-1})\notin [0,1]$ and $(2^{n-1}-2^{n}v+1)\notin [0,1]$.  Now compute
\begin{align*}
&(1-v) a+v b-(1-v)(\sqrt{a}-\sqrt{b})^{2}-(2v-1)\sqrt{ab}\left\{\sum_{k=2}^n2^{k-2}\left(\sqrt[2^{k}]{\frac{b}{a}}-1\right)^{2}\right\}\\
&=(1-v) a+v b-(1-v)(\sqrt{a}-\sqrt{b})^{2}\\
&\qquad -(2v-1)\sqrt{ab}\left\{\left(\sqrt[4]{\frac{b}{a}}-1\right)^{2}+2\left(\sqrt[8]{\frac{b}{a}}-1\right)^{2}+4\left(\sqrt[16]{\frac{b}{a}}-1\right)^{2}\right\}\\
&\qquad -\cdots -2^{n-4}(2v-1)\sqrt{ab}\left(\sqrt[2^{n-2}]{\frac{b}{a}}-1\right)^{2}\\
&\qquad -2^{n-3}(2v-1)\sqrt{ab}\left(\sqrt[2^{n-1}]{\frac{b}{a}}-1\right)^{2} -2^{n-2}(2v-1)\sqrt{ab}\left(\sqrt[2^{n}]{\frac{b}{a}}-1\right)^{2}\\
&=(1-v) a+v b-(1-v)(a-2\sqrt{ab}+b)-(2v-1)\sqrt{ab}\left(\sqrt{\frac{b}{a}}-2\sqrt[4]{\frac{b}{a}}+1\right)\\
&\qquad -2(2v-1)\sqrt{ab}\left(\sqrt[4]{\frac{b}{a}}-2\sqrt[8]{\frac{b}{a}}+1\right)-4(2v-1)\sqrt{ab}\left(\sqrt[8]{\frac{b}{a}}-2\sqrt[16]{\frac{b}{a}}+1\right)\\
&\qquad -\cdots -2^{n-4}(2v-1)\sqrt{ab}\left(\sqrt[2^{n-3}]{\frac{b}{a}}-2\sqrt[2^{n-2}]{\frac{b}{a}}+1\right)\\
&\qquad -2^{n-3}(2v-1)\sqrt{ab}\left(\sqrt[2^{n-2}]{\frac{b}{a}}-2\sqrt[2^{n-1}]{\frac{b}{a}}+1\right)\\
&\qquad -2^{n-2}(2v-1)\sqrt{ab}\left(\sqrt[2^{n-1}]{\frac{b}{a}}-2\sqrt[2^{n}]{\frac{b}{a}}+1\right)\\
&=2(1-v)\sqrt{ab}+(1-2v)\sqrt{ab}\sum_{l=0}^{n-2}2^l+2^{n-1}(2v-1)\sqrt{ab}\sqrt[2^{n}]{\frac{b}{a}}\\
&=\left\{2(1-v)+(1-2v)\sum_{l=0}^{n-2}2^l \right\}\sqrt{ab}+2^{n-1}(2v-1)\sqrt{ab}\sqrt[2^{n}]{\frac{b}{a}}\\
&=\left\{2(1-v)+(1-2v)(2^{n-1}-1) \right\}\sqrt{ab}+2^{n-1}(2v-1)\sqrt{ab}\sqrt[2^{n}]{\frac{b}{a}}\\
&=(2^{n-1}-2^{n}v+1)\sqrt{ab}+(2^{n}v-2^{n-1})\sqrt{ab}\sqrt[2^{n}]{\frac{b}{a}}\\
&\leq \left(\sqrt{ab}\right)^{(2^{n-1}-2^{n}v+1)}\left( \sqrt{ab}\sqrt[2^{n}]{\frac{b}{a}} \right)^{(2^{n}v-2^{n-1})}\: \text{( by Lemma \ref{l0})}\\
&=a^{1-v}v^{v}.
\end{align*}
 So we get the following inequality
\begin{align*}
&(1-v) a+v b-(1-v)(\sqrt{a}-\sqrt{b})^{2}-(2v-1)\sqrt{ab}\left\{\sum_{k=2}^n2^{k-2}\left(\sqrt[2^{k}]{\frac{b}{a}}-1\right)^{2}\right\}\\
& \qquad \leq a^{1-v}v^{v}
\end{align*}
which is equivalent to \eqref{eq5}.

\textit{(ii)} If  $v \notin [\frac{2^{n-1}-1}{2^n},\frac{1}{2}]$, then $(1-v)\notin [\frac{1}{2},\frac{2^{n-1}+1}{2^n}]$.  Now by changing two elements $a,b$ and replacing the weight $v$ with $(1-v)$ in \textit{(i)}, the desired inequality \eqref{eq6} is deduced.\qed
\end{proof}

\begin{remark}
We would remark that if we rewrite Theorem \ref{t2} for $n=1$, then we get Corollary \ref{c2}.
\end{remark}

\begin{remark}\label{RM}
From the equality of the proof in Theorem \ref{t2}, the inequality \eqref{eq5} is equivalent to
\begin{align}\label{ineq0002}
 a^{1-v} b^v &\geq \sqrt{ab}+2^{n}\left(v-\frac{1}{2}\right)\sqrt{ab} \left(\sqrt[2^n]{\frac{b}{a}}-1\right),
\end{align}
which gives the following inequality
\begin{align*}
 \left( \frac{b}{a} \right)^{v-\frac{1}{2}}\geq 1+2^n\left(v - \frac{1}{2} \right) \left(\sqrt[2^n]{\frac{b}{a}}-1  \right).
\end{align*}
Since $\lim_{r\to 0} \frac{t^r -1}{r} =\log t$, by putting $r=\frac{1}{2^{n}}$, we have
\begin{align*}
\lim_{n\to \infty}2^{n}\left(\sqrt[2^n]{\frac{b}{a}}-1\right) = \log \frac{b}{a}.
\end{align*}
Thus, we have the following inequality in the limit of $n \to \infty$ for the inequality \eqref{eq5} in Theorem \ref{t2}:
\begin{equation} \label{ineq0001}
\log \left( \frac{b}{a} \right)^{v-\frac{1}{2}} \leq \left( \frac{b}{a} \right)^{v-\frac{1}{2}}-1,
\end{equation}
for $a,b > 0$ and $\frac{1}{2} \neq v \in \mathbb{R}$, which comes from the condition $v \notin \left[\frac{1}{2},\frac{2^{n-1}+1}{2^n} \right]$ in the limit of $n \to \infty$. The above inequality recover the equality in the case $v =\frac{1}{2}$.
Therefore, we have the inequality \eqref{ineq0001} for all $v \in \mathbb{R}$. We notice that the inequality \eqref{ineq0001} can be proven directly by putting $x=\left(\frac{b}{a}\right)^{v-1/2}$ in the inequality 
\begin{equation} \label{fundamental_log_ineq}
\log x \leq x-1,\quad (x>0).
\end{equation}
Similarly in the limit of $n \to \infty$ for the inequality \eqref{eq6} in Theorem \ref{t2}, we get the inequality 
$$
\log \left( \frac{a}{b} \right)^{\frac{1}{2}-v} \leq \left( \frac{a}{b} \right)^{\frac{1}{2}-v}-1,
$$
 by  changing two elements $a,b$ and replacing the weight $v$ with $(1-v)$ in the inequality \eqref{ineq0001}.
\end{remark}

Next, in Remarks \ref{r1} and  \ref{r2}, we show that Theorem \ref{t2} recover the inequalities in Lemma \ref{l1}. To achieve this, we compare  Lemma \ref{l1} with Theorem \ref{t2}  in the cases such as $n=2$ and $n=3$ where $v\in[0,1]$. First, we notice that  Lemma \ref{l1} is equivalent to the following proposition.

\begin{prop}\label{p3}
Let $a,b\geq 0$ and $v\in [0,1]$.
\begin{enumerate}[(i)]
\item If  $v\in[0,\frac{1}{4}]$, then
\begin{align*}
(1-v)a+vb &\leq a^{1-v}b^{v}+(1-v)(\sqrt{a}-\sqrt{b})^{2}-2v(\sqrt{b}-\sqrt[4]{ab})^{2}.
\end{align*}
\item If $v\in[\frac{1}{4},\frac{1}{2}]$, then
\begin{align*}
(1-v)a+vb &\leq a^{1-v}b^{v}+(1-v)(\sqrt{a}-\sqrt{b})^{2}+(2v-1)(\sqrt{b}-\sqrt[4]{ab})^{2}.
\end{align*}
\item If $v\in[\frac{1}{2},\frac{3}{4}]$, then
\begin{align*}
(1-v)a+vb &\leq a^{1-v}b^{v}+v(\sqrt{a}-\sqrt{b})^{2}-(2v-1)(\sqrt{a}-\sqrt[4]{ab})^{2}.
\end{align*}
\item If $v\in[\frac{3}{4},1]$, then
\begin{align*}
(1-v)a+vb &\leq a^{1-v}b^{v}+v(\sqrt{a}-\sqrt{b})^{2}+(2v-2)(\sqrt{a}-\sqrt[4]{ab})^{2}.
\end{align*}
\end{enumerate}
\end{prop}

\begin{remark}\label{r1}
Consider Theorem \ref{t2} in the case $n=2$ with $v\in [0,1]$.  For $a,b> 0$, we have the following inequalities
\begin{enumerate}[\textit{(i)}]
\item If $v\notin[\frac{1}{2},\frac{3}{4}]$, then
\begin{align}
(1-v)a+vb &\leq a^{1-v}b^{v}+(1-v)(\sqrt{a}-\sqrt{b})^{2}\notag \\
& \qquad +(2v-1)(\sqrt{b}-\sqrt[4]{ab})^{2}.\label{eq7}
\end{align}
\item If $v\notin[\frac{1}{4},\frac{1}{2}]$, then
\begin{align}
(1-v)a+vb &\leq a^{1-v}b^{v}+v(\sqrt{a}-\sqrt{b})^{2}\notag \\
& \qquad  -(2v-1)(\sqrt{a}-\sqrt[4]{ab})^{2}.\label{eq8}
\end{align}
\end{enumerate}
 In our recent paper \cite{xx}, we showed that the right-hand sides of both inequalities (\ref{eq7}) and (\ref{eq8}) give tighter upper bounds of the $v$-weighted arithmetic mean than those in Proposition \ref{p3}.
\end{remark}

\begin{remark}\label{r2}
As a direct consequence of Theorem \ref{t2} in the case $n=3$ with restricted range $v\in[0,1]$, we have
\begin{enumerate}[\textit{(i)}]
\item If $v\notin[\frac{1}{2},\frac{5}{8}]$, then
\begin{align}\label{eq9}
(1-v)a+vb &\leq a^{1-v}b^{v}+(1-v)(\sqrt{a}-\sqrt{b})^{2}+(2v-1)(\sqrt{b}-\sqrt[4]{ab})^{2}\notag \\
&\qquad +(4v-2)(\sqrt[8]{ab^{3}}-\sqrt[4]{ab})^{2}.
\end{align}
\item If $v\notin[\frac{3}{8},\frac{1}{2}]$, then
\begin{align}\label{eq10}
(1-v)a+vb &\leq a^{1-v}b^{v}+v(\sqrt{a}-\sqrt{b})^{2}-(2v-1)(\sqrt{a}-\sqrt[4]{ab})^{2}\notag \\
&\qquad -(4v-2)(\sqrt[8]{a^{3}b}-\sqrt[4]{ab})^{2}.
\end{align}
\end{enumerate}
We here give advantages of inequalities \eqref{eq9} and \eqref{eq10} in comparison with Proposition \ref{p3}.
\begin{itemize}
\item[(a)] Firstly, we compare Proposition \ref{p3} with the inequality \eqref{eq9} which holds  in the cases $v \in [0,\frac{1}{4}]$, $v \in [\frac{1}{4},\frac{1}{2}]$, $v \in [\frac{5}{8},\frac{3}{4}]$ and $v \in [\frac{3}{4},1]$ .
\begin{itemize}
\item[(a1)] In the case $v \in [0,\frac{1}{4}]$, we have $(2v-1)<(-2v)$ and $(4v-2)<0$. Indeed, the right-hand side of inequality \eqref{eq9} is less than the right-hand side of \textit{(i)} in Proposition \ref{p3}. For the case of $v \in [\frac{1}{4},\frac{1}{2}]$, we have $(4v-2)<0$. So we easily find that the right-hand side of inequality \eqref{eq9} is less than the right-hand side of \textit{(ii)} in Proposition \ref{p3}. 

\item[(a2)]For the case of $v \in [\frac{5}{8},\frac{3}{4}]$, we claim that the right-hand side of inequality \eqref{eq9} is less than or equal to the right-hand side of \textit{(iii)} in Proposition \ref{p3}. To prove our claim, we show that the following inequality holds:
\begin{align}\label{eq100}
 &(1-v)(\sqrt{a}-\sqrt{b})^{2}+(2v-1)(\sqrt{b}-\sqrt[4]{ab})^{2}+(4v-2)(\sqrt[8]{ab^{3}}-\sqrt[4]{ab})^{2} \notag \\
&\qquad \leq v(\sqrt{a}-\sqrt{b})^{2}-(2v-1)(\sqrt{a}-\sqrt[4]{ab})^{2},
\end{align}
 which is equivalent to the inequality 
\begin{equation*}
 2(1-2v) t^{1/4}\left\{3t^{1/4}-1-2t^{3/8}\right\}\geq 0,
\end{equation*}
for $t>0$ and $v \in [\frac{5}{8},\frac{3}{4}]$. To obtain this, it is enough to prove $(3t^{1/4}-1-2t^{3/8})\leq 0$. If $t^{1/8}=x$, then we easily find that $f(x)=3x^2-2x^3-1$  is increasing for $0<x<1$ and decreasing where $x>1$. Indeed, $f(x)=3x^2-2x^3-1\leq 0$ where $x>0$ and so $3t^{1/4}-1-2t^{3/8} \leq 0$.

\item[(a3)]For the case of  $v \in [\frac{3}{4},1]$, we  claim that the right-hand side of inequality \eqref{eq9} is less than or equal to the right-hand side of \textit{(iv)} in Proposition \ref{p3}. To prove our claim, we show that the following inequality holds:
\begin{align}\label{eq101}
 &(1-v)(\sqrt{a}-\sqrt{b})^{2}+(2v-1)(\sqrt{b}-\sqrt[4]{ab})^{2}+(4v-2)(\sqrt[8]{ab^{3}}-\sqrt[4]{ab})^{2}\notag \\
&\qquad \leq v(\sqrt{a}-\sqrt{b})^{2}-2(1-v)(\sqrt{a}-\sqrt[4]{ab})^{2},
\end{align}
 which is equivalent to the inequality 
\begin{equation}\label{eq11}
 (4v-3)+(3-8v)t^{1/2}+(4-4v)t^{1/4}+(8v-4)t^{5/8}\geq 0,
\end{equation}
for $t>0$ and $v \in [\frac{3}{4},1]$. To obtain the inequality \eqref{eq11}, it is sufficient to prove $f(x,v) \geq 0$ where $ x=t^{1/8} >0$ and
$$
f(x,v) \equiv (8v-4)x^5 +(3-8v) x^4 +(4-4v) x^2 +(4v-3).
$$

Since  $\frac{df(x,v)}{dv} = 4(x-1)^2 (2x^3+2x^2+2x+1) \geq 0$, we have $f(x,v) \geq f(x,\frac{3}{4}) =x^2 (x-1)^2(2x+1) \geq 0$ for $x>0$.
\end{itemize}

\item[(b)] Secondly, we compare Proposition \ref{p3} with the inequality \eqref{eq10} which holds  in the cases $v \in [0,\frac{1}{4}]$, $v \in [\frac{1}{4},\frac{3}{8}]$, $v \in [\frac{1}{2},\frac{3}{4}]$ and $v \in [\frac{3}{4},1]$.
\begin{itemize}
\item[(b1)]For the case of  $v \in [0,\frac{1}{4}]$, we claim that the right-hand side of inequality \eqref{eq10} is less than or equal to the right-hand side of \textit{(i)} in Proposition \ref{p3}. To prove our claim, we give the following inequality
\begin{align*}
 &v(\sqrt{a}-\sqrt{b})^{2}-(2v-1)(\sqrt{a}-\sqrt[4]{ab})^{2}-(4v-2)(\sqrt[8]{a^{3}b}-\sqrt[4]{ab})^{2}\\
&\qquad \leq (1-v)(\sqrt{a}-\sqrt{b})^{2}-2v(\sqrt{b}-\sqrt[4]{ab})^{2},
\end{align*}
which we get it by replacing $v$ with $(1-v)$ and changing the elements $a,b$ in the inequality \eqref{eq101} .

\item[(b2)] For the case of $v \in [\frac{1}{4},\frac{3}{8}]$, we  claim that the right-hand side of inequality \eqref{eq10} is less than or equal to the right-hand side of \textit{(ii)} in Proposition \ref{p3}. To prove our claim, we give the following inequality
\begin{align*}
 &v(\sqrt{a}-\sqrt{b})^{2}-(2v-1)(\sqrt{a}-\sqrt[4]{ab})^{2}-(4v-2)(\sqrt[8]{a^{3}b}-\sqrt[4]{ab})^{2}\\
&\qquad \leq (1-v)(\sqrt{a}-\sqrt{b})^{2}-(1-2v)(\sqrt{b}-\sqrt[4]{ab})^{2},
\end{align*}
which is deduced by replacing $v$ with $(1-v)$ and changing the elements $a,b$ in the inequality \eqref{eq100} .

\item[(b3)] For the case of $v \in [\frac{1}{2},\frac{3}{4}]$, we have $-(4v-2)<0$. So we easily find that the right-hand side of inequality \eqref{eq10} is less than the right-hand side of \textit{(iii)} in Proposition \ref{p3}. In the case $v \in [\frac{3}{4},1]$ , we have  $-(2v-1)<(2v-2)$ and $-(4v-2)<0$. That is the right-hand side of inequality \eqref{eq10} is less than the right-hand side of \textit{(iv)} in Proposition \ref{p3}.
\end{itemize}
\end{itemize}
\end{remark}

Thus,  according to Remark \ref{r1} and Remark \ref{r2}, Theorem \ref{t2} recover Lemma \ref{l1}.  We notice that the range of the reverse Young's inequalities in Theorem \ref{t2} is wider than Lemma \ref{l1}, namely Theorem \ref{t2} holds in the case $v \in \mathbb{R}$ and Lemma \ref{l1} holds for $v \in [0,1]$. \\

Next, we compare Theorem \ref{t2} with Lemma \ref{SM}. In Theorem \ref{t2}, we have (I) $v\leq 0$, (II) $0\leq v < \frac{2^{n-1}-1}{2^n}$, (III) $\frac{2^{n-1}-1}{2^n} \leq v \leq \frac{1}{2}$, (IV) $\frac{1}{2} \leq v \leq \frac{2^{n-1}+1}{2^n}$, (V) $\frac{2^{n-1}+1}{2^n} < v \leq 1$ and (VI) $v\geq 1$.\\
For the cases of (II) and (V), we claim that Theorem \ref{t2} has tighter upper bounds than those in Lemma \ref{SM}, while Lemma \ref{SM} recover Theorem \ref{t2} in the cases (III) and (IV). Therefore we conclude that Theorem \ref{t2} and  Lemma \ref{SM} are different refinements of the reverse Young's inequality which both of them recover Lemma \ref{l1}. However, we emphasize that Theorem \ref{t2} gives the reverse Young's inequality in the wider range than Lemma \ref{SM}, namely  in the case $v \in \mathbb{R}$. This justify why our refinement in Theorem \ref{t2} is better than Lemma \ref{SM}.\\
 To prove our claims, we compare Theorem \ref{t2} with Lemma \ref{SM} in the same steps such as  $n=2$. For this purpose, we list up the following corollaries which are deduced directly from Theorem \ref{t2} and Lemma \ref{SM}, respectively.
\begin{corollary}\label{c1}
 Let $a,b> 0$ and $v\in \mathbb{R}$.
\begin{enumerate}[(i)]
\item If $v\notin[\frac{1}{2},\frac{3}{4}]$, then
\begin{align}
(1-v)a+vb - a^{1-v}b^{v} &\leq (1-v)(\sqrt{a}-\sqrt{b})^{2}+(2v-1)(\sqrt{b}-\sqrt[4]{ab})^{2}.\label{eq19}
\end{align}
\item If $v\notin[\frac{1}{4},\frac{1}{2}]$, then
\begin{align}
(1-v)a+vb - a^{1-v}b^{v} &\leq v(\sqrt{a}-\sqrt{b})^{2}-(2v-1)(\sqrt{a}-\sqrt[4]{ab})^{2}.\label{eq20}
\end{align}
\end{enumerate}
\end{corollary}

\begin{corollary}\label{c11}
Let $a,b> 0$ and $v\in[0,1]$.
\begin{enumerate}[(i)]
\item If $v\in [0,\frac{1}{4}]$, then
\begin{align}
(1-v)a+vb - a^{1-v}b^{v} &\leq (1-v)(\sqrt{a}-\sqrt{b})^{2}-2v(\sqrt{b}-\sqrt[4]{ab})^{2}\notag \\ &\qquad -4v(\sqrt{b}-\sqrt[8]{ab^3})^{2}.\label{eq15}
\end{align}
\item If $v\in[\frac{1}{4},\frac{1}{2}]$, then
\begin{align}
(1-v)a+vb - a^{1-v}b^{v} &\leq (1-v)(\sqrt{a}-\sqrt{b})^{2}+(2v-1)(\sqrt{b}-\sqrt[4]{ab})^{2}\notag \\ &\qquad-(4v-1)(\sqrt[8]{ab^3} -\sqrt[4]{ab})^{2}.\label{eq16}
\end{align}
\item If $v\in[\frac{1}{2},\frac{3}{4}]$, then
\begin{align}
(1-v)a+vb - a^{1-v}b^{v} &\leq v(\sqrt{a}-\sqrt{b})^{2}-(2v-1)(\sqrt{a}-\sqrt[4]{ab})^{2}\notag \\ &\qquad+(4v-3)(\sqrt[8]{ab^3} -\sqrt[4]{ab})^{2}.\label{eq17}
\end{align}
\item If $v\in[\frac{3}{4},1]$, then
\begin{align}
(1-v)a+vb - a^{1-v}b^{v} &\leq v(\sqrt{a}-\sqrt{b})^{2}+(2v-2)(\sqrt{a}-\sqrt[4]{ab})^{2}\notag \\ &\qquad-4(1-v)(\sqrt[8]{ab^3} -\sqrt{a})^{2}.\label{eq18}
\end{align}
\end{enumerate}
\end{corollary}

\begin{remark}
Here, we compare the upper bounds in Corollary \ref{c1} with those in Corollary \ref{c11}. Firstly, we compare Corollary \ref{c11} with the inequality (\ref{eq19}) which holds in the cases $v\in[0,\frac{1}{4}]$, $v\in[\frac{1}{4},\frac{1}{2}]$ and $v\in[\frac{3}{4},1]$.\\
For the case of $v\in[0,\frac{1}{4}]$, we can find examples such that the right-hand side of (\ref{eq19}) is tighter than that of (\ref{eq15}) in Corollary \ref{c11}. Actually, take $a=1$, $b=16$ and $v=1/8$, then the right-hand side of (\ref{eq19}) is  equal to $4.875$, while the right-hand side of (\ref{eq15}) is nearly equal to $6.2892$. Indeed,  the inequality (\ref{eq19}) can recover Corollary \ref{c11} where $v\in[0,\frac{1}{4}]$.\\
In the case $v\in[\frac{3}{4},1]$, we claim that the right-hand side of the inequality (\ref{eq19}) is less than or equal to the right-hand side of (\ref{eq18}). According to (\ref{eq101}), we have
\begin{align*}
 &(1-v)(\sqrt{a}-\sqrt{b})^{2}+(2v-1)(\sqrt{b}-\sqrt[4]{ab})^{2}\notag \\
&\qquad \leq v(\sqrt{a}-\sqrt{b})^{2}+(2v-2)(\sqrt{a}-\sqrt[4]{ab})^{2}-(4v-2)(\sqrt[8]{ab^{3}}-\sqrt[4]{ab})^{2}.
\end{align*}
So to prove our claim, we show that the following inequality holds
\begin{align*}
(2-4v)(\sqrt[8]{ab^{3}}-\sqrt[4]{ab})^{2}\leq (4v-4)(\sqrt[8]{ab^{3}}-\sqrt{a})^{2},
\end{align*}
which is equivalent to 
\begin{align*}
g(x,v)=(4v-2)x^6+(2-4v)x^5+(2v-1)x^4+(2-4v)x^3+(2v-1)\geq 0,
\end{align*}
where $x=t^{\frac{1}{8}}>0$ and $v\in[\frac{3}{4},1]$. Since $\frac{dg(x,v)}{dv}=2(x-1)^2(2x^4+2x^3+3x^2+2x+1)\geq 0$, we have $g(x,v)\geq g(x,\frac{3}{4})=(x-1)^2(x^4+x^3+\frac{3}{2}x^2+x+\frac{1}{2})\geq 0$. This justify that the inequality (\ref{eq19}) also recover Corollary \ref{c11} where $v\in[\frac{3}{4},1]$. However, in the case  $v\in[\frac{1}{4},\frac{1}{2}]$, we easily find that the right-hand side of (\ref{eq16}) is less than or equal to the right-hand side of (\ref{eq19}). That is Corollary \ref{c11} is a refinement of (\ref{eq19}) where $v\in[\frac{1}{4},\frac{1}{2}]$.\\
Secondly, we compare Corollary \ref{c11} with the inequality (\ref{eq20}) which holds in the cases $v\in[0,\frac{1}{4}]$, $v\in[\frac{1}{2},\frac{3}{4}]$ and $v\in[\frac{3}{4},1]$. The comparison is done by the same way as in the first step, and we omit it. 
\end{remark}

Sababheh and Choi gave the following refinement of Lemma \ref{l0} which it's complete proof can be found in \cite[Theorem 2.2]{SM}.

\begin{lemma}\textsc{( \cite[Theorem 2.9]{SC})}\label{t00}
 Let $a,b>0$. Then, we have
 \begin{enumerate}[(i)]
 \item If $v\leq 0$, then
 \begin{align}\label{SC1}
(1-v) a+v b \leq a^{1-v} b^v + v \sum_{k=1}^n 2^{k-1} \left( \sqrt{a} -\sqrt[2^k]{a^{2^{k-1}-1}b}\right)^2.
\end{align}
 \item If $v\geq 1$, then
\begin{align}\label{SC2}
(1-v) a+v b \leq a^{1-v} b^v + (1-v) \sum_{k=1}^n 2^{k-1} \left( \sqrt{b} -\sqrt[2^k]{ab^{2^{k-1}-1}}\right)^2.
\end{align}
\end{enumerate}
\end{lemma}

We can extend the ranges of $v$ in Lemma \ref{t00} to those in following theorem,
by the similar way to the line of  proof of Theorem \ref{t2}.

\begin{theorem}\label{t11}
Let $a,b >0$, $n \in \mathbb{N}$ and $v \in \mathbb{R}$.
\begin{enumerate}[(i)]
\item If $v \notin [0,\frac{1}{2^n}]$, then
the inequality \eqref{SC1} holds.

\item If $v \notin [\frac{2^{n}-1}{2^n},1]$, then
the inequality \eqref{SC2} holds.
\end{enumerate}
\end{theorem}

As we discussed the case of $n \to \infty$ in Remark \ref{RM},
we also give the following remark.
\begin{remark}\label{RM2}
The inequality (\ref{SC1}) is equivalent to the inequality
$$
a+v a 2^n\left( \left( \frac{b}{a}\right)^{1/2^n} -1\right) \leq a^{1-v} b^v
$$
by the elementary computations. Using the formula $\lim_{r\to 0} \frac{t^r-1}{r} = \log t$, we have the inequality
$$
\log\left( \frac{b}{a} \right)^v \leq \left( \frac{b}{a} \right)^v -1
$$
for $v \neq 0$ in the limit of $n \to \infty$. The above inequality trivially holds for all $v \in \mathbb{R}$. It can be also obtained by the inequality (\ref{fundamental_log_ineq}).

Similarly, the inequality (\ref{SC2}) is equivalent to the inequality
$$
b+(1-v)b2^n \left( \left( \frac{a}{b}\right)^{1/2^n} -1\right) \leq a^{1-v} b^v
$$
so that we have  the inequality
$$
\log\left( \frac{a}{b} \right)^{1-v }\leq \left( \frac{a}{b} \right)^{1-v} -1
$$
for $v \neq 1$ in the limit of $n \to \infty$. The above inequality trivially holds for all $v \in \mathbb{R}$. It can be also obtained by the inequality (\ref{fundamental_log_ineq}).
\end{remark}

As a direct consequence of Theorems \ref{t2} and \ref{t11}, we have the following  reverse inequalities with respect to the Heinz means.

\begin{corollary}\label{c00}
 Let $a,b > 0$, $n \in \mathbb{N}$ such that $n\geq 2$ and $\frac{1}{2}\neq v\in \mathbb{R}$.
\begin{enumerate}[(i)]
\item If $v \notin [\frac{1}{2},\frac{2^{n-1}+1}{2^n}]$, then
\begin{align*}
\frac{a+b}{2} &\leq H_{v}(a,b)+(1-v)(\sqrt{a}-\sqrt{b})^{2}\notag\\
&\qquad +\left(v-\frac{1}{2}\right)\sqrt{ab}\sum_{k=2}^n2^{k-2}\left\{\left(\sqrt[2^{k}]{\frac{a}{b}}-1\right)^{2}+\left(\sqrt[2^{k}]{\frac{b}{a}}-1\right)^{2}  \right\}
\end{align*}

\item If $v \notin [\frac{2^{n-1}-1}{2^n},\frac{1}{2}]$, then
\begin{align*}
\frac{a+b}{2} &\leq H_{v}(a,b)+v(\sqrt{a}-\sqrt{b})^{2}\notag\\
&\qquad +\left(\frac{1}{2}-v\right)\sqrt{ab}\sum_{k=2}^n2^{k-2}\left\{\left(\sqrt[2^{k}]{\frac{a}{b}}-1\right)^{2}+\left(\sqrt[2^{k}]{\frac{b}{a}}-1\right)^{2}  \right\}.
\end{align*}
\end{enumerate}
\end{corollary}

\begin{corollary}\label{c0011}
 Let $a,b > 0$, $n \in \mathbb{N}$ and $v \in \mathbb{R}$.
\begin{enumerate}[(i)]
\item If $v \notin [0,\frac{1}{2^n}]$, then
\begin{align*}
\frac{a+b}{2} &\leq H_{v}(a,b)\\
&\qquad + v \sum_{k=1}^n 2^{k-2} \left\lbrace \left( \sqrt{a} -\sqrt[2^k]{a^{2^{k-1}-1}b}\right)^2+\left( \sqrt{b} -\sqrt[2^k]{ab^{2^{k-1}-1}}\right)^2\right\rbrace .
\end{align*}

\item If $v \notin [\frac{2^{n}-1}{2^n},1]$, then
\begin{align*}
\frac{a+b}{2} &\leq H_{v}(a,b)\\
&\,\,+ (1-v) \sum_{k=1}^n 2^{k-2} \left\lbrace \left( \sqrt{a} -\sqrt[2^k]{a^{2^{k-1}-1}b}\right)^2+\left( \sqrt{b} -\sqrt[2^k]{ab^{2^{k-1}-1}}\right)^2\right\rbrace .
\end{align*}
\end{enumerate}
\end{corollary}

\section{Generalized Reverse Young and Heinz Inequalities for Operators}
In this section by applying Kubo--Ando theory \cite{nn} and thanks to Theorems \ref{t2} and  \ref{t11}, we have the following operator inequalities.

\begin{theorem}\label{t6}
Let $A, B\in B^{++}(H)$,  $n \in \mathbb{N}$ such that $n\geq 2$ and $\frac{1}{2}\neq v\in \mathbb{R}$. Then, we have the following inequalities.
\begin{enumerate}[(i)]
\item If $v \notin [\frac{1}{2},\frac{2^{n-1}+1}{2^n}]$, then
\begin{align*}
A\nabla_{v} B &\leq A\natural_{v} B+2(1-v)(A\nabla_{v} B-A\sharp B)\\
&\qquad +(2v-1)\sum_{k=2}^{n}2^{k-2}\left( A\sharp B-2A\sharp_{\frac{2^{k-1}+1}{2^k}}B+A\sharp_{\frac{2^{k-2}+1}{2^{k-1}}}B\right).
\end{align*}

\item If $v \notin [\frac{2^{n-1}-1}{2^n},\frac{1}{2}]$, then
\begin{align*}
A\nabla_{v} B &\leq A\natural_{v} B+2v(A\nabla_{v} B-A\sharp B)\\
&\qquad +(1-2v)\sum_{k=2}^{n}2^{k-2}\left( A\sharp B-2A\sharp_{\frac{2^{k-1}-1}{2^k}}B+A\sharp_{\frac{2^{k-2}-1}{2^{k-1}}}B\right).
\end{align*}
\end{enumerate}
\end{theorem}

\begin{proof}

\textit{(i)} According to the inequality  \eqref{eq5}, the following inequality holds for $t\geq 0$
\begin{align*}
(1-v)+v t &\leq  t^v+(1-v)(1-\sqrt{t})^{2}\notag\\
&\qquad +(2v-1)\sqrt{t}\sum_{k=2}^n2^{k-2}\left(\sqrt[2^{k}]{t}-1\right)^{2}.
\end{align*}
By functional calculus, if we replace $t$ with  $A^{-\frac{1}{2}}BA^{-\frac{1}{2}}$ and then multiplying both sides of the inequality by $A^{\frac{1}{2}}$, the desired inequality is obtained.

\textit{(ii)} The line of proof is similar to \textit{(i)} by applying  the inequality  \eqref{eq6}.\qed
\end{proof}

\begin{remark}
 We notice that Theorem \ref{t6} with $v\in[0,1]$,  recover the inequalities obtained in \cite[Theorem 3]{xx}, if we put $n=2$.
\end{remark}

\begin{theorem}\label{t6-6}
Let $A, B\in B^{++}(H)$, $n \in \mathbb{N}$ and $v \in \mathbb{R}$. Then, we have the following inequalities.
\begin{enumerate}[(i)]
\item  If $v \notin [0,\frac{1}{2^n}]$, then
\begin{align*}
A\nabla_{v} B &\leq A\natural_{v} B +v \sum_{k=1}^n 2^{k-1} \left( A-2A\sharp_{\frac{1}{2^k}}B+A\sharp_{\frac{1}{2^{k-1}}}B\right).
\end{align*}

\item If $v \notin [\frac{2^{n}-1}{2^n},1]$, then
\begin{align*}
A\nabla_{v} B &\leq A\natural_{v} B +(1-v) \sum_{k=1}^n 2^{k-1} \left( B-2A\sharp_{\frac{2^k-1}{2^k}}B+A\sharp_{\frac{2^{k-1}-1}{2^{k-1}}}B\right).
\end{align*}
\end{enumerate}
\end{theorem}

\begin{proof}
\textit{(i)} According to Theorem \ref{t11} in the case $v \notin [0,\frac{1}{2^n}]$, the following inequality holds for $t\geq 0$
\begin{align*}
(1-v)+v t &\leq t^v + v \sum_{k=1}^n 2^{k-1} \left( 1 -\sqrt[2^k]{t}\right)^2\\
&= t^v + v \sum_{k=1}^n 2^{k-1} \left( 1-2t^{\frac{1}{2^k}} +t^{\frac{1}{2^{k-1}}}\right).
\end{align*}
By functional calculus, if we replace $t$ with  $A^{-\frac{1}{2}}BA^{-\frac{1}{2}}$ and then multiplying both sides of the inequality by $A^{\frac{1}{2}}$, the desired inequality is deduced.

\textit{(ii)} By applying Theorem \ref{t11} for the case of $v \notin [\frac{2^{n}-1}{2^n},1]$, we get the desired inequality in the same way as in \textit{(i)}. \qed
\end{proof}

As a direct consequence of Theorems \ref{t6} and  \ref{t6-6}, we get the generalized reverse Heinz operator inequalities as follows. That is Corollaries \ref{c3} and  \ref{c3-3} are operator versions of Corollaries \ref{c00} and \ref{c0011} respectively.

\begin{corollary}\label{c3}
Let $A, B\in B^{++}(H)$, $n \in \mathbb{N}$ such that $n\geq 2$ and $\frac{1}{2}\neq v\in \mathbb{R}$. Then we have the following inequalities.
\begin{enumerate}[(i)]
\item If $v \notin [\frac{1}{2},\frac{2^{n-1}+1}{2^n}]$, then
\begin{align*}
A\nabla B &\leq \hat{H}_{v}(A,B) +2(1-v)(A\nabla B-A\sharp B)\\
&\quad +(2v-1)\sum_{k=2}^{n}2^{k-2}\left( A\sharp B-2H_{\frac{2^{k-1}+1}{2^k}}(A,B)+H_{\frac{2^{k-2}+1}{2^{k-1}}}(A,B)\right).
\end{align*}

\item If $v \notin [\frac{2^{n-1}-1}{2^n},\frac{1}{2}]$, then
\begin{align*}
A\nabla B &\leq \hat{H}_{v}(A,B) +2v(A\nabla B-A\sharp B)\\
&\quad +(1-2v)\sum_{k=2}^{n}2^{k-2}\left( A\sharp B-2H_{\frac{2^{k-1}-1}{2^k}}(A,B)+H_{\frac{2^{k-2}-1}{2^{k-1}}}(A,B)\right).
\end{align*}
\end{enumerate}
\end{corollary}

\begin{remark}
 Putting $n=2$ in Corollary \ref{c3} with  $v\in[0,1]$ gives the inequalities obtained in \cite[Corollary 6]{xx}.
\end{remark}

\begin{corollary}\label{c3-3}
Let $A, B\in B^{++}(H)$, $n \in \mathbb{N}$ and $v \in \mathbb{R}$. Then, we have the following inequalities.
\begin{enumerate}[(i)]
\item  If $v \notin [0,\frac{1}{2^n}]$, then
\begin{align*}
A\nabla B &\leq \hat{H}_{v}(A,B)\\ &\qquad +v \sum_{k=1}^n 2^{k-1} \left( A\nabla B-2H_{\frac{1}{2^k}}(A,B)+H_{\frac{1}{2^{k-1}}}(A,B)\right).
\end{align*}

\item If $v \notin [\frac{2^{n}-1}{2^n},1]$, then
\begin{align*}
A\nabla B &\leq \hat{H}_{v}(A,B)\\
& \quad +(1-v) \sum_{k=1}^n 2^{k-1} \left( A\nabla B-2H_{\frac{2^k-1}{2^k}}(A,B)+H_{\frac{2^{k-1}-1}{2^{k-1}}}(A,B)\right).
\end{align*}
\end{enumerate}
\end{corollary}

\begin{remark}
If we put $n=1$, Corollary \ref{c3-3} gives the inequality shown in \cite[Theorem 3.1]{aa} where $v \notin [\frac{1}{2},1]$.
\end{remark}

\begin{acknowledgements}
The authors express their gratitude to the editor-in-chief  Prof. Rosihan M. Ali and the anonymous referees for their careful reading and detailed comments which have considerably improved the paper.
\end{acknowledgements}



\end{document}